\documentclass[a4paper,10pt]{article}

\usepackage{authblk}

\usepackage[utf8]{inputenc} 
\usepackage[T1]{fontenc}    
\usepackage{indentfirst}
\usepackage{lmodern}

\usepackage{mathtools}
\usepackage{verbatim}
\usepackage{url}            
\usepackage{booktabs}       
\usepackage{amsfonts}       
\usepackage{amsmath}
\usepackage{amssymb}
\usepackage{amsthm}
\usepackage{microtype}      
\usepackage[a4paper]{geometry} 
\usepackage{enumitem}
\usepackage{xcolor}

\usepackage[linktocpage=true]{hyperref}
\hypersetup{
    colorlinks=true,
    linkcolor=blue,
    urlcolor=blue
}

\usepackage{graphicx}
\usepackage[makeroom]{cancel}
\usepackage{tikz}

\usepackage[inkscapeformat=png]{svg}
\usepackage[normalem]{ulem}

\title{Invertibility in partially ordered nonassociative rings}
\author{Nizar El Idrissi* \\
nizar.elidrissi@uit.ac.ma \and Hicham Zoubeir* \\
hzoubeir2014@gmail.com}

\newcommand{\Addresses}{{
  \bigskip
  \footnotesize

  \textbf{Nizar El Idrissi.}
  \par\nopagebreak Laboratoire : Equations aux dérivées partielles, Algèbre et Géométrie spectrales.
  \par\nopagebreak
  Département de mathématiques, faculté des sciences, université Ibn Tofail, 14000 Kénitra.\par\nopagebreak 
  \textit{E-mail address} : \texttt{nizar.elidrissi@uit.ac.ma}

  \medskip

  \textbf{Hicham Zoubeir.} 
  \par\nopagebreak Laboratoire : Equations aux dérivées partielles, Algèbre et Géométrie spectrales.
  \par\nopagebreak
   Département de mathématiques, faculté des sciences, université Ibn Tofail, 14000 Kénitra.
\par\nopagebreak 
  \textit{E-mail address} : \texttt{hzoubeir2014@gmail.com}
}}

\theoremstyle{plain}
\newtheorem{theorem}{Theorem}[section]
\newtheorem{proposition}{Proposition}[section]
\newtheorem{corollary}{Corollary}[section]
\newtheorem{lemma}{Lemma}[section]

\theoremstyle{definition}
\newtheorem{definition}{Definition}[section]
\newtheorem{example}{Example}[section]   

\theoremstyle{remark}
\newtheorem{remark}{Remark}[section]

\makeatletter
\def\keywords{\xdef\@thefnmark{}\@footnotetext}
\makeatother

\begin{document}
\newpage
\maketitle


\begin{abstract}
Invertibility is important in ring theory because it enables division and facilitates solving equations. Moreover, (nonassociative) rings can be endowed with an extra ''structure'' such as order and topology allowing more richness in the theory. The two main theorems of this article are contributions to invertibility in the context of partially ordered nonassociative rings \textit{and} Hausdorff sequentially Cauchy-complete weak-quasi-topological nonassociative rings. Specifically, the first theorem asserts that the interval $]0,1]$ in any suitable partially ordered nonassociative ring consists entirely of invertible elements. The second theorem asserts that if $f$ is a suitably generalized concept of seminorm from a nonassociative ring to a partially ordered nonassociative ring endowed with Frink's interval topology, then under certain conditions, the subset of elements such that $f(1-a) < 1$ consists entirely of invertible elements. Part of the assumption of the second theorem is that of Hausdorff sequential Cauchy-completeness of the first ring under the topology induced by the seminorm $f$ (which takes values in a partially ordered nonassociative ring endowed with Frink's interval topology). Frink's interval topology is an example of a coarse locally-convex $T_1$ topology. Moreover, to our knowledge, the topology induced by a seminorm into a partially ordered nonassociative ring has never been introduced. Some additional original facts, such as the fact that the topology on a nonassociative ring $R_1$ induced by a norm into a totally ordered associative division ring $R_2$ endowed with Frink's interval topology (or equivalently, with the order topology, since the order of $R_2$ is total) is a Hausdorff locally convex quasi-topological group with an additional separate continuity property of the product, are dealt with in the second section ''Preliminaries''.
\end{abstract}




\def\thefootnote{*}\footnotetext{These authors contributed equally to this work.}


\tableofcontents

\section{Introduction}

Ring theory is a rich and deep mathematical subject \cite{Matsumura,Lam,Schafer}. Rings are, after groups, among the most pervasive structures in algebra, and therefore in mathematics. Like many other mathematical structures, rings have a diverse nature in that they can support different levels of extra ''structure'' (ordered ring, topological ring, algebra, bialgebra, Hopf algebra, etc.) and ''property'' (associativity, commutativity, supercommutativity, etc.). Rings have two operations, a ''+'' and a ''$\times$''. Generally, the second operation is nontrivial. Highly interesting examples of rings have been found (or constructed) in mathematics. Since the second operation is often the most interesting, an important direction in ring theory is what is named ''multiplicative ideal theory'' \cite{Chapman-Fontana-Geroldinger-Olberding}, and in particular, the theory of inverses \cite{Rao}. Invertibility is important in ring theory because it enables division, facilitates the solving of equations, allows the definition of a group of units, and plays a role in ideal theory. The questions that arise in this context are of the following type: 1) how can the inverses be defined? 2) Do inverses exist and are they unique? 3) What is the size of the set of invertible elements? Some definitions and answers are given in \cite{Kaplansky,Kantrowitz-Neumann2017,Drazin1958,Drazin2012}. But while the preceding explanations give an indication of the richness of this field, this article contributes exclusively to the answer to the last two questions, using the usual notion of invertibility, and working in the specific additional context of partially ordered and topological structures. Certainly, topology is a well-established theory. Briefly, it consists of inspecting the nearly geometric properties of sets through the examination of the properties of a collection of their subsets. In his landmark paper \cite{Kaplansky}, Kaplansky showed that rings endowed with a compatible topology enjoy some very rich properties. Since then, many contributions have been made to the theory of topological rings \cite{Warner, Arnautov-Glavatsky-Mikhalev}. Kaplansky also found a result, on which we base our work, stating the invertibility of elements of rings under a condition involving what was previously called a metric, and using an argument involving geometric series. This particular result was developed in \cite{Garcia-Pacheco-Miralles-Murillo-Arcila, Kantrowitz-Neumann2017, Kantrowitz-Neumann2015, Propp, Clark-Diepeveen}, but is still leaving room for improvement. On the other hand, order theory is a powerful theory that uses the notion of an ''order relation'' (a binary relation) to explore the properties of sets. This subject started to be mainstream in the '40's, with the book of Birkhoff \cite{Birkhoff} serving as an important foundation and source of future investigations. The book has been modernized several times, and many important notions from order theory were gradually discussed in this book, such as lattices and topologies on partially ordered sets (posets). However, one of the first topologies to be introduced on posets, the interval topology, was by Frink \cite{Frink1942}. Soon afterwards, Birkhoff, Frink and McShane \cite{Birkhoff, Frink1942, Mcshane} (who were all colleagues) introduced the concept of o-convergence. Since then, there has been an abundance of scientific publications in this field.  We now know that a poset may be endowed with many possible topologies (the order topology, the lower topology, the upper topology, Frink's interval topology \cite{Frink1942}, the ideal topology \cite{Frink1954}, the Lawson topology \cite{Gierz-Hofmann-Keimel-Lawson-Mislove-Scott}, the Scott topology \cite{Gierz-Hofmann-Keimel-Lawson-Mislove-Scott, Amadio-Curien}, the Alexandroff topology, etc.) or notions of convergence ($o1$-convergence \cite{Birkhoff, Frink1942, Mcshane, Mathews-Anderson}, $o2$-convergence \cite{Rennie,Ward, Mathews-Anderson}, $o3$-convergence \cite{Mathews-Anderson}, $uo$-convergence, $\lim-\inf$ convergence, MN-convergence \cite{Sun-Li-Fan}, etc.), each one having its own peculiarities. For example, the Scott topology is more often used when studying DCPO in domain theory \cite{Gierz-Hofmann-Keimel-Lawson-Mislove-Scott, Amadio-Curien}, whereas the $uo$-convergence relation is only applied within lattices.

When rings are used alongside the notions of order and topology, we may expect to arrive at the notion of a ''topological ordered ring'' \cite{Garcia-Pacheco-Moreno-Frias-Murillo-Arcila, Branga-Olaru}.
The authors in \cite{Garcia-Pacheco-Moreno-Frias-Murillo-Arcila} explored this direction and showed, for example, that the order topology need not be abandoned in favor of other topologies as it can endow some rings with a valuable topological ordered ring structure. However, we must emphasize that this article is about partially ordered nonassociative rings \cite{Steinberg} \textit{and} weak-quasi-topological nonassociative rings. More precisely, our article is organized mainly around two theorems. First, we prove that the interval $]0,1]$ in any suitable partially ordered nonassociative ring consists entirely of invertible elements. In the second one, we show that if $f$ is a suitably generalized concept of seminorm from a nonassociative ring to a partially ordered nonassociative ring endowed with Frink's interval topology, then under certain conditions, the subset of elements such that $f(1-a) < 1$ consists entirely of invertible elements. Part of the assumption of the second theorem is that of Hausdorff sequential Cauchy-completeness of the first ring under the topology induced by the seminorm $f$ (which takes values in an ordered nonassociative ring endowed with Frink's interval topology). We also dedicate the second section ''Preliminaries'' to show some additional original facts, such as: a nonassociative unitary ring endowed with the topology induced by a norm into a totally ordered associative division ring endowed with Frink's interval topology is a Hausdorff locally convex quasi-topological group with an additional separate continuity property of the product. In this work, such a ring which is quasi-topological as a group and has the additional separate continuity property of the product is called a weak-quasi-topological ring. Finally, we feel obligated to cite the interesting work \cite{Kopperman-Pajoohesh} in which the authors have introduced, in section 5, a very similar topology with a similar degree of generality. 

\section{Preliminaries}

\subsection{Notation}

We set up some notation. \\
$\mathbb{N}$ is the set of natural numbers including 0. \\
For each $n \in \mathbb{N}$, $[\![1,n]\!]$ stands for the set $\{1,2,\cdots,n\}$ of natural numbers. If $n=0$, then this set is just the empty set. \\
If $(P,\leq)$ is a partially ordered set, then $<$ stands for the associated strict order: $a<b$ iff $a \leq b$ and $a \neq b$. \\
Lastly, the notations $[a,b]$, $[a,b[$, $]a,b]$, and $]a,b[$ in a poset $(P,\leq)$ stand for the closed interval of elements $x$ such that $a \leq x \leq b$, the semi-open interval of elements $x$ such that $a \leq x < b$, the semi-open interval of elements $x$ such that $a < x \leq b$, and the open interval of elements $x$ such that $a < x < b$, respectively. 

\subsection{Basic concepts}

\begin{definition}
Let $(P, \leq)$ be a partially ordered set. A subset $S \subseteq P$ is \textbf{convex} if 
\[ \forall x,y,z \in P : (x \in S) \wedge (z \in S) \wedge (x \leq y \leq z) \Rightarrow y \in S. \]
\end{definition}

\begin{definition}
A \textbf{nonassociative and noncommutative ring} $(R,+,0,\times)$ is a set $R$ endowed with an element $0 \in R$ and two binary operations $+ : R \times R \to R$ and $\times : R \times R \to R$ such that $(R,+,0)$ is an Abelian group and $\times$ is left and right distributive with respect to $+$ (i.e., $\times$ is biadditive). \\
For brevity, we call such a structure simply a \textbf{ring}.
\end{definition}

\begin{definition}
\label{definition-ordered-group}
A \textbf{partially ordered group} $(G,\times,1,\leq)$ is a group with a \textit{compatible} partial order. Here, the compatibility means:
\[ \forall x \in G : \forall y \geq z : [(xy \geq xz) \text{ and } (yx \geq zx)]. \]
\end{definition}

\begin{definition}
A \textbf{partially ordered ring} $(R,+,0,\times,\leq)$ is a ring endowed with a \textit{compatible} partial order. Here, the compatibility means that $(R,+,0,\leq)$ is an ordered Abelian group, and that if $R^+ := \{x \in R : x \geq 0\}$, then $(R^+)^2 := (R^+)(R^+) \subseteq R^+$. \\
If $R$ has a multiplicative unit $1$, we may also require that $1 \geq 0$.
\end{definition}

\begin{definition}
Let $(G_1,\times,1)$ be a group and $(G_2,\times,1,\leq)$ be a partially ordered group. \\ 
An \textbf{even submultiplicative function} from $G_1$ to $G_2$ is any map $f$ from $G_1$ to $G_2$ satisfying:
\begin{itemize}
\item $\forall x,y \in R_1 : f(xy) \leq f(x)f(y)$ (submultiplicativity/triangle inequality),
\item $\forall x \in G_1 : f(x^{-1}) = f(x)$ (even function).
\end{itemize}
If $f$ also satisfies:
\begin{itemize}
\item $\forall x \in G_1 : f(x) \geq 0$ (nonnegativity),
\end{itemize}
then $f$ is called a \textbf{seminorm}, and if $f$ satisfies furthermore
\begin{itemize}
\item $\forall x \in G_1 : f(x) = 0 \Rightarrow x=0$ (definiteness),
\end{itemize}
then $f$ is called a \textbf{norm}.
\end{definition}

\begin{definition}
Let $(R_1,+,0,\times)$ be a ring and $(R_2,+,0,\times,\leq)$ be a partially ordered ring. \\ 
An \textbf{even subadditive and submultiplicative function} from $R_1$ to $R_2$ is any map $f$ from $R_1$ to $R_2$ satisfying:
\begin{itemize}
\item $\forall x,y \in R_1 : f(x+y) \leq f(x)+f(y)$ (subadditivity/triangle inequality),
\item $\forall x \in R_1 : f(-x) = f(x)$ (even function),
\item $\forall x,y \in R_1 : f(xy) \leq f(x)f(y)$ (submultiplicativity),
\item if $R_1$ and $R_2$ are unitary, $f(1) \leq 1$.
\end{itemize}
If $f$ also satisfies:
\begin{itemize}
\item $\forall x \in R_1 : f(x) \geq 0$ (nonnegativity),
\end{itemize}
then $f$ is called a \textbf{seminorm}, and if $f$ satisfies furthermore
\begin{itemize}
\item $\forall x \in R_1 : f(x) = 0 \Rightarrow x=0$ (definiteness),
\end{itemize}
then $f$ is called a \textbf{norm}.
\end{definition}

\begin{definition}
Let $(G,\times,1)$ be a group endowed with a nonnecessarily compatible topology $\tau$. \\
We say that $(u_n)_{n \in \mathbb{N}}$ $(u_n \in G \text{ for all } n \in \mathbb{N})$ is a \textbf{Cauchy sequence} if for any any neighborhood $V \subseteq G$ of $1 \in G$, there exists $N \in \mathbb{N}$ such that for all $n,m \geq N$ : $u_n u_m^{-1} \in V$ and $u_n^{-1} u_m \in V$.
\end{definition}

A poset can be endowed with many possible topologies or notions of convergence. For the purpose of this article, we focus on Frink's interval topology. A poset $(P,\leq)$ endowed with the interval topology forms a $T_1$ space.

\begin{definition}
Let $(P,\leq)$ be a poset. Then $P$ can be endowed with a topology, called \textbf{Frink's interval topology}, whose subbase of closed sets is formed by the union of the sets of the type $\{x \in P : x \geq a \}$ and $\{x \in P : x \leq b\}$ where $a$ and $b$ vary in $P$.
\end{definition}

When the order is total, the interval topology is the same as the following topology:

\begin{definition}
Let $(P,\leq)$ be a poset. Then $P$ can be endowed with a topology, called the \textbf{order topology} (not to be confused with the topology associated with order convergence), whose subbase of open sets is formed by the union of the sets of the type $\{x \in P : x > a \}$ and $\{x \in P : x < b\}$ where $a$ and $b$ vary in $P$.
\end{definition}

Some authors have shown that the order topology is still valuable in posets \cite{Garcia-Pacheco-Moreno-Frias-Murillo-Arcila} even though closed intervals are not generally closed for this topology.

The following properties are three examples of ''completeness'' properties. The Archimedean property is considered a completeness property since in a totally ordered field, it is equivalent to the existence of $\inf\limits_{n \in \mathbb{N}} \frac{1}{n+1}$ and the equality $\inf\limits_{n \in \mathbb{N}} \frac{1}{n+1} = 0$ (\cite{Teismann}, Proposition 4 p.105).

\begin{definition}
Let $(G,\times,1,\leq)$ be a partially ordered group. \\
We say that $G$ is \textbf{Archimedean} if
\[ \forall a,b \in G : \left( \forall n \in \mathbb{Z} : a^n \leq b \right) \Rightarrow a \leq 1. \]
A patially ordered ring $(R,\leq)$ is said to be \textbf{Archimedean} if $(R,+,0,\leq)$ is Archimedean as a partially ordered group.
\end{definition}

\begin{example}
We have:
\begin{itemize}
\item The rings $\mathbb{Z}, \mathbb{Q}, \mathbb{R},$ and $\mathbb{C}$ are Archimedean.
\item The rings $\mathbb{R}(x)$ and $\mathbb{R}((x))$ with the partial ordering $F = \sum\limits_{k=v}^d c_kX^k \geq 0$ iff $F=0$ or $c_v > 0$ are not Archimedean. Here, we have $d \in \mathbb{Z}$ in the $\mathbb{R}(x)$ case and $d \in \mathbb{Z} \cup \{+\infty\}$ in the $\mathbb{R}((x))$ case.
\end{itemize}
\end{example}

\begin{definition}
Let $(P,\leq)$ be a poset. We say that $(P,\leq)$ is \textbf{monotone $\sigma$-complete} if any increasing bounded above sequence in $P$ has a supremum.
\end{definition}

\begin{definition}
Let $(G,\times,1)$ be a group endowed with a nonnecessarily compatible topology $\tau$. We say that $G$ is \textbf{sequentially Cauchy complete} if any Cauchy sequence in $G$ converges to at least one limit.
\end{definition}

\subsection{Topology induced by a seminorm and applications}
\label{subsection-properties-topology}

\begin{definition}
\label{definition-quasi-topological-group}
A \textbf{quasi-topological group} is a group $(G,\times,1)$ endowed with a topology $\tau$ satisfying
\[
\begin{aligned}
&\bullet \forall V \in \tau : \forall a \in G : aV \in \tau \text{ and } Va \in \tau, \\
&\bullet \forall V \in \tau : V^{-1} \in \tau,
\end{aligned}
\]
\end{definition}

Equivalently, a quasi-topological group must satisfy

\[
\label{Eq3}
\begin{aligned}
&\bullet \cdot \times \cdot : \begin{cases} G \times G &\to G \\ (x,y) &\mapsto xy \end{cases} \text{ is separately continuous}, \\
&\bullet (\cdot)^{-1} : \begin{cases} G &\to G \\ x &\mapsto x^{-1} \end{cases} \text{ is continuous}.
\end{aligned}
\]

and equivalently, that for any net $(u_\alpha)_{\alpha \in (A,\leq)}$ in $G$ converging to $u \in G$,
\[
\label{Eq}
\begin{aligned}
&(u_\alpha u^{-1})_{\alpha \in (A,\leq)} \text{ converges to 1 and } (u^{-1} u_\alpha)_{\alpha \in (A,\leq)} \text{ converges to 1}, \\
&(u_\alpha^{-1})_{\alpha \in (A,\leq)} \text{ converges to } u^{-1}.
\end{aligned}
\]

\begin{definition}
A \textbf{locally convex quasi-topological partially ordered group} $(G,\times,1,\tau, \leq)$ is a quasi-topological group endowed with a compatible partial order (in the sense of definition \ref{definition-ordered-group}) and admitting a base of convex open sets. 
\end{definition}

\begin{definition}
Let $(G_1,\times,1)$ be a group and $(G_2,\times,1,\tau,\leq)$ be a locally convex quasi-topological partially ordered group. \\
Let $f : (G_1,\times,1) \to (G_2,\times,1,\leq)$ be an even submultiplicative function. \\
In this case, $G_1$ can be endowed with the \textbf{topology induced by $f$}, a topology whose base (see proposition \ref{proposition-base-topology-induced-by-f}) of open sets is $B_V(g) := \{ x \in G_1 : f(xg^{-1}) \in V\}$ where $g$ runs through $G_1$ and where $V$ runs through the neighborhoods of $1$ in $G_2$. This topology is more precisely \textbf{the right topology induced by $f$}, where ''right'' accounts for the possible noncommutativity of $x$ and $g^{-1}$ in the definition of $B_V(g)$.
\end{definition}

\begin{proposition}
\label{proposition-base-topology-induced-by-f}
Let $(G_1,\times,1)$ be a group and $(G_2,\times,1,\tau,\leq)$ be a locally convex quasi-topological partially ordered group. \\
Let $f : (G_1,\times,1) \to (G_2,\times,1,\leq)$ be an even submultiplicative function. \\
Then the sets $B_V(g) := \{ x \in G_1 : f(xg^{-1}) \in V\}$ where $g$ runs through $G_1$ and $V$ runs through the neighborhoods of $1$ in $G_2$, form a base of a topology on $G_1$, called the \textbf{right topology induced by $f$}.
\end{proposition}

\begin{proof}
We need to show that for any $B_V(g)$ and $B_{V'}(g')$ such that $B_V(g) \cap B_{V'}(g') \neq \emptyset$, there exist $V''$ and $g''$ such that $B_{V''}(g'') \subseteq B_V(g) \cap B_{V'}(g')$. \\
Let $g'' \in B_V(g) \cap B_{V'}(g')$.  First, note that we can assume without loss of generality that $V$ and $V'$ are convex, since $G_2$ is a locally convex quasi-topological partially ordered group.  Next, observe that since $G_2$ is a quasi-topological group,
\[ V'' := (V f(g''g^{-1})^{-1}) \cap (V f(g''g^{-1})^{-1})^{-1} \cap (V' f(g''g'^{-1})^{-1}) \cap (V' f(g''g'^{-1})^{-1})^{-1} \]
is a neighborhood of $1$ in $G_2$. \\
Let $x \in B_{V''}(g'')$. We have 
\[ f(xg''^{-1})^{-1} f(g''g^{-1}) \leq f(xg^{-1}) \leq  f(xg''^{-1}) f(g''g^{-1}). \]
Therefore, there exist $v_1$ and $v_2$ in $V$ such that
\[ v_1 \leq f(xg^{-1}) \leq v_2. \]
By convexity of $V$, we obtain $f(xg^{-1}) \in V$, so $B_{V''}(g'') \subseteq B_V(g)$. \\
Similarly, $B_{V''}(g'') \subseteq B_{V'}(g')$, and the proof is finished.
\end{proof}

\begin{corollary}
\label{corollary-open-topology-associated-to-f}
Let $(G_1,\times,1)$ be a group and $(G_2,\times,1,\leq)$ be a locally convex quasi-topological partially ordered group. \\
Let $f : (G_1,\times,1) \to (G_2,\times,1,\leq)$ be an even submultiplicative function. \\
Then a set $E \subseteq G_1$ is open $\Leftrightarrow \forall e \in E : \exists V_e \text{ neighborhood of } 1 : B_{V_e}(e) \subseteq E.$
\end{corollary}

\begin{proof}
The proof of the corollary in the direction ($\Leftarrow$) is clear. In the other direction, the idea is to note that for any $B_V(g)$ and $B_{V'}(g')$ such that $B_V(g) \cap B_{V'}(g') \neq \emptyset$, $g''$ was chosen arbitrarily from $B_V(g) \cap B_{V'}(g')$ in the proof of proposition \ref{proposition-base-topology-induced-by-f}. Therefore, let $E \subseteq G_1$ be an open susbset. For any $e \in E$, since $E$ is an open subset of $G_1$ containing $e$, we can choose $g \in G_1$ and an open neighborhood $V$ of $1$ in $G_2$ such that $B_V(g)$ contains $e$ and is included in $E$. By choosing $V'=V$ and $g'=g$, we can deduce by proposition \ref{proposition-base-topology-induced-by-f} the existence of an open neighborhood $V_e$ of $1$ such that $B_{V_e}(e) \subseteq B_V(g) \cap B_V(g) = B_V(g) \subseteq E$.
\end{proof}

\begin{proposition}
Let $(G_2,\times,1,\leq)$ be a locally convex quasi-topological partially ordered group. \\
Moreover, let $(G_1,\times,1)$ be an abelian group and $f : (G_1,\times,1) \to (G_2,\times,1,\leq)$ be an even submultiplicative function. \\
We endow $G_1$ with the topology induced by $f$. \\
Then $G_1$ is a quasi-topological group.
\end{proposition} 

\begin{proof}
An open subset $W$ of $G_1$ has the form 
\[ \bigcup\limits_{e \in E} B_{V_e}(g_e), \]
where $V_e$ are open sets in $G_2$ containing $1$ and $g_e \in G_1$. \\
To prove that the assertions of definition \ref{definition-quasi-topological-group} hold for $W$, note that it is sufficient to prove that they hold for $B_{V_e}(g_e)$ for any choice of $e \in E$, because 
\begin{align*}
& a\left( \bigcup\limits_{e \in E} B_{V_e}(g_e) \right) = \bigcup\limits_{e \in E} (aB_{V_e}(g_e))  \\
& \left( \bigcup\limits_{e \in E} B_{V_e}(g_e) \right)a = \bigcup\limits_{e \in E} (B_{V_e}(g_e)a) \\
& \left( \bigcup\limits_{e \in E} B_{V_e}(g_e) \right)^{-1} = \bigcup\limits_{e \in E} (B_{V_e}(g_e))^{-1} 
\end{align*}
and an arbitrary union of open subsets is an open subset. \\
Since $G_1$ is abelian, we have for any $e \in E$,
\[ aB_{V_e}(g_e) = B_{V_e}(g_e)a, \]
and
\begin{align*}
y \in aB_{V_e}(g_e) &\Leftrightarrow ya^{-1} \in B_{V_e}(g_e)) \\
&\Leftrightarrow f(ya^{-1}g_e^{-1}) \in V_e \\
&\Leftrightarrow f(y(g_ea)^{-1}) \in V_e \\
&\Leftrightarrow y \in B_{V_e}(g_ea)
\end{align*}
So 
\[
aB_{V_e}(g_e) = B_{V_e}(g_e)a = B_{V_e}(g_ea).
\]
is an open subset of $G_1$. \\
Moreover, using again the fact that $G_1$ is abelian, we have
\begin{align*}
y \in B_{V_e}(g_e)^{-1} &\Leftrightarrow y^{-1} \in B_{V_e}(g_e)) \\
&\Leftrightarrow f(y^{-1}g_e^{-1}) \in V_e \\
&\Leftrightarrow f(g_e^{-1}y^{-1}) \in V_e \\
&\Leftrightarrow f((yg_e)^{-1}) \in V_e \\
&\Leftrightarrow f(yg_e) \in V_e \\
&\Leftrightarrow y \in B_{V_e}(g_e^{-1})
\end{align*}
So
\[
B_{V_e}(g_e)^{-1} = B_{V_e}(g_e^{-1}).
\]
is an open subset of $G_1$.
\end{proof}

The following proposition shows that multiplication by $a$ (where $f(a) \leq 1$) is continuous when $R_1$ is endowed with the topology induced by a seminorm into a locally convex quasi-topological partially ordered ring. Rings that are quasi-topological as groups and which furthermore satisfy this property (or its stronger form stated in proposition \ref{proposition-phia-continuous}) are termed weak-quasi-topological rings in this article.

\begin{proposition}
\label{proposition-phia-continuous-f(a)-<=-1}
Let $(R_1,+,0,\times)$ be a ring and $(R_2,+,0,\times,\leq)$ be a locally convex quasi-topological partially ordered ring.
Let $f : (R_1,+,0,\times) \to (R_2,+,0,\times,\leq)$ be a seminorm. \\
We endow $R_1$ with the topology induced by $f$. \\
Then for all $a \in R_1$ such that $f(a) \leq 1$, the map $\phi_a : \begin{cases} R_1 &\to R_1 \\ x &\mapsto ax \end{cases}$ is continuous.
\end{proposition}

\begin{proof}
Let $r \in R_1$. Let $W$ be a neighborhood of $ar$ in $R_1$. By corollary \ref{corollary-open-topology-associated-to-f} and local convexity of $R_2$, there exists a convex neighborhood $V$ of $0$ in $R_2$ such that 
\[ B_V(ar) \subseteq W. \]
Let $x \in B_{V}(r)$. We have 
\[ 0 \leq f(ax - ar) \leq  f(x-r). \]
Therefore, since $0$ and $f(x-r)$ belong to $V$, it follows by convexity of $V$ that $f(ax-ar) \in V$, so that $\phi_a(x) \in B_{V}(ar) \subseteq W$. Hence, $\phi_a$ is continuous at $r \in R_1$, which holds for any $r \in R_1$.
\end{proof}

When $R_2$ is assumed to be an associative division ring, we have more:

\begin{proposition}
\label{proposition-phia-continuous}
Let $(R_1,+,0,\times)$ be a ring and $(R_2,+,0,\times,\leq)$ be a locally convex quasi-topological partially ordered division ring.
Let $f : (R_1,+,0,\times) \to (R_2,+,0,\times,\leq)$ be a seminorm. \\
We endow $R_1$ with the topology induced by $f$. \\
Then for all $a \in R_1$, the map $\phi_a : \begin{cases} R_1 &\to R_1 \\ x &\mapsto ax \end{cases}$ is continuous.
\end{proposition}

\begin{proof}
Let $r \in R_1$. Let $W$ be a neighborhood of $ar$ in $R_1$. By corollary \ref{corollary-open-topology-associated-to-f} and the local convexity of $R_2$, there exists a convex neighborhood $V$ of $0$ in $R_2$ such that 
\[ B_V(ar) \subseteq W. \]
Because $R_2$ is a division ring, there exists a neighborhood $V'$ of $0$ such that
\[ f(a)V' \subseteq V. \]
Let $x \in B_{V'}(r)$. We have 
\[ 0 \leq f(ax - ar) \leq f(a)f(x-r). \]
Therefore, since $0$ and $f(a)f(x-r)$ belong to $V$, it follows by convexity of $V$ that $f(ax-ar) \in V$, so $\phi_a(x) \in B_{V}(ar) \subseteq W$. Hence, $\phi_a$ is continuous at $r \in R_1$, which holds for any $r \in R_1$.
\end{proof}

\subsection{The interval topology as an example of a locally convex poset topology}

The following proposition shows that all the properties of section \ref{subsection-properties-topology} are true for $(G_2,\times,1)$, a partially ordered group endowed with the interval topology.

\begin{proposition}
Let $(G,\times,1,\leq)$ be a partially ordered group that we endow with the interval topology. Then, $G$ is a locally convex quasi-topological group.
\end{proposition}

\begin{proof}
\begin{itemize}
\item Let $V$ be an open subset of $G$, and let $a \in G$. To prove that the assertions of definition \ref{definition-quasi-topological-group} hold for $V$, note that it is sufficient to prove that they hold when $V$ is an element of the base of open subsets of $G$, because
\begin{align*}
& a\left( \bigcup\limits_{e \in E} V_e \right) = \bigcup\limits_{e \in E} (aV_e)  \\
& \left( \bigcup\limits_{e \in E} V_e \right)a = \bigcup\limits_{e \in E} (V_ea) \\
& \left( \bigcup\limits_{e \in E} V_e \right)^{-1} = \bigcup\limits_{e \in E} (V_e)^{-1} 
\end{align*}
and an arbitrary union of open subsets is an open subset. \\
Therefore, let $V$ be an element of the base of open subsets of $G$. We can write
\[
V = G \setminus \left[ \left( \bigcup\limits_{i \in [\![1,n_1]\!]} ]-\infty,b_i] \right) \cup \left( \bigcup\limits_{j \in [\![1,n_2]\!]} [a_j,+\infty[ \right) \right].
\]
$x \in a V$ is equivalent to $a^{-1}x \in V$, which is equivalent to
\[
\begin{aligned}
&\forall i \in [\![1,n_1]\!] : \neg(a^{-1}x \leq b_i), \\
&\forall j \in [\![1,n_2]\!] : \neg(a_j \leq a^{-1}x)
\end{aligned}
\]
and is also equivalent to
\[
\begin{aligned}
&\forall i \in [\![1,n_1]\!] : \neg(x \leq ab_i), \\
&\forall j \in [\![1,n_2]\!] : \neg(aa_j \leq x)
\end{aligned}
\]
Therefore
\[
\begin{aligned}
aV =  G \setminus \left[ \left( \bigcup\limits_{i \in [\![1,n_1]\!]} ]-\infty,ab_i] \right) \cup \left( \bigcup\limits_{j \in [\![1,n_2]\!]} [aa_j,+\infty[ \right) \right]
\end{aligned}
\]
is an open subset of $G$.
Similarly
\[
\begin{aligned}
Va =  G \setminus \left[ \left( \bigcup\limits_{i \in [\![1,n_1]\!]} ]-\infty,b_ia] \right) \cup \left( \bigcup\limits_{j \in [\![1,n_2]\!]} [a_ja,+\infty[ \right) \right]
\end{aligned}
\]
is an open subset of $G$.
Moreover, $x \in V^{-1}$ is equivalent to $x^{-1} \in V$, which is equivalent to
\[
\begin{aligned}
&\forall i \in [\![1,n_1]\!] : \neg(x^{-1} \leq b_i), \\
&\forall j \in [\![1,n_2]\!] : \neg(a_j \leq x^{-1})
\end{aligned}
\]
and is also equivalent to
\[
\begin{aligned}
&\forall i \in [\![1,n_1]\!] : \neg(b_i^{-1} \leq x), \\
&\forall j \in [\![1,n_2]\!] : \neg(x \leq a_j^{-1})
\end{aligned}
\]
Therefore
\[
\begin{aligned}
V^{-1} =  G \setminus \left[ \left( \bigcup\limits_{j \in [\![1,n_2]\!]} ]-\infty,a_j^{-1}] \right) \cup \left( \bigcup\limits_{i \in [\![1,n_1]\!]} [b_i^{-1},+\infty[ \right) \right].
\end{aligned}
\]
is an open subset of $G$.
\item Let us prove that $G$ is locally convex, i.e., that $G$ admits a base of convex open subsets. Let us use the defining base of open subsets of $G$. An open subset of this base is written as follows:
\[
V = G \setminus \left[ \left( \bigcup\limits_{i \in [\![1,n_1]\!]} ]-\infty,b_i] \right) \cup \left( \bigcup\limits_{j \in [\![1,n_2]\!]} [a_j,+\infty[ \right) \right].
\]
Let $v_1$ and $v_2$ be two elements of $V$, and $x \in G$ is such that 
\[ v_1 \leq x \leq v_2. \]
For the sake of contradiction, suppose that $x \notin V$. Then, we have either
\begin{itemize}
\item $\exists i \in [\![1,n_1]\!] : x \leq b_i$, which leads to the contradiction $v_1 \leq b_i$.
\item $\exists j \in [\![1,n_2]\!] : a_j \leq x$, which leads to the contradiction $a_j \leq v_2$.
\end{itemize}
Therefore, $x \in V$ and $V$ is convex; therefore $G$ is locally convex.
\end{itemize}
\end{proof}

\subsection{Existence and uniqueness of limits in Frink's interval topology}

The following lemma is also needed in the remainder of the article.

\begin{lemma}
\label{lemma-increasing-sequence-converges-to-sup-interval-topology}
Let $(P,\leq)$ be a poset and $(u_n)_{n \in \mathbb{N}}$ be an increasing sequence of elements of $P$ admitting a supremum $u := \sup\limits_{n \in \mathbb{N}} \{u_n\}$. \\
Then $u$ is the unique limit of $(u_n)_{n \in \mathbb{N}}$ in Frink's interval topology.
\end{lemma}

\begin{proof}
\begin{itemize}
\item Let us show that $u$ is a limit of $(u_n)_{n \in \mathbb{N}}$. It suffices to show that for any open subset of the form
\[ V = P \setminus \left[ \left( \bigcup\limits_{i \in [\![1,n_1]\!]} ]-\infty,b_i] \right) \cup \left( \bigcup\limits_{j \in [\![1,n_2]\!]} [a_j,+\infty[ \right) \right] \]
containing $u$, there exists $N \in \mathbb{N}$ such that for all $n \geq N$, $u_n \in V$.
We have $u \in V$, so:
\[
\begin{aligned}
&\forall i \in [\![1,n_1]\!] : \neg(u \leq b_i), \\
&\forall j \in [\![1,n_2]\!] : \neg(a_j \leq u)
\end{aligned}
\]
Let $n \in \mathbb{N}$. Suppose by contradiction that $u_n \notin V$. Then we have either
\begin{itemize}
\item $\exists j \in [\![1,n_2]\!] : a_j \leq u_n$. In this case, we have $u \geq u_n \geq a_j$, so $u \geq a_j$, a contradiction.
\item $\exists i \in [\![1,n_1]\!] : u_n \leq b_i$. 
\end{itemize}
Since $[\![1,n_1]\!]$ is finite and $\mathbb{N}$ is infinite, there exists $i \in [\![1,n_1]\!]$ and a subsequence $(u_{\phi(n)})_{n \in \mathbb{N}}$ where $\phi : \mathbb{N} \to \mathbb{N}$ is a strictly increasing sequence, such that
\[ \forall n \in \mathbb{N} : u_{\phi(n)} \leq b_i \]
Since $(u_n)_{n \in \mathbb{N}}$ is increasing, it follows that
\[ \forall n \in \mathbb{N} : u_n \leq b_i \]
Passing to the supremum, we have $u \leq b_i$, which is a contradiction. \\
Therefore, there exists $N \in \mathbb{N}$ such that $u_n \in V$. Since the interval topology is locally convex and $(u_n)_{n \in \mathbb{N}}$ increases to $u$, it follows that 
\[ \forall n \geq N : u_n \in V. \]
\item Let us prove that $u$ is the unique limit of $(u_n)_{n \in \mathbb{N}}$. Let $u'$ be another limit of $(u_n)_{n \in \mathbb{N}}$, and suppose that $u' \neq u$. \\
Since $u' \neq u$, we have $\neg(u' \leq u)$ or $\neg(u \leq u')$. \\
Suppose by contradiction that $\neg(u' \leq u)$. Then the open subset
\[ V' = P \setminus ]-\infty,u] \]
is a neighborhood of $u'$. Since $u'$ is a limit of $(u_n)_{n \in \mathbb{N}}$, there exists $N \in \mathbb{N}$ such that for all $n \geq N$, $u_n \in V'$, that is, $\neg(u_n \leq u)$, a contradiction. Therefore, $\neg(u \leq u')$. \\
Suppose that $u_n \leq u'$ for an infinite number of indices $n \in \mathbb{N}$. Since $(u_n)_{n \in \mathbb{N}}$ increases, it follows that 
\[ \forall n \in \mathbb{N} : u_n \leq u' \]
which implies that $u \leq u'$, which is a contradiction. Since $\mathbb{N}$ is infinite, there exists therefore $m \in \mathbb{N}$ such that $\neg(u_m \leq u')$. Then the open subset
\[ V'' = P \setminus [u_m,+\infty[ \]
is a neighborhood of $u'$. Since $u'$ is a limit of $(u_n)_{n \in \mathbb{N}}$, there exists $N \in \mathbb{N}$ such that for all $n \geq N$, $u_n \in V''$, that is, $\neg(u_m \leq u_n)$. Taking $n := \max(N,m)$, we see that this is a contradiction. \\
Hence the limit of $(u_n)_{n \in \mathbb{N}}$ is unique.
\end{itemize}
\end{proof}

\subsection{Totally ordered groups or rings and topology}

The following proposition is proved in \cite{Olivos-Soto-Mansilla} and has been known perhaps even before that date. A minor reformulation of the proof is proposed below.

\begin{proposition}
\label{proposition-totally-ordered-group-topological-group}
Let $(G,\leq)$ be a totally ordered group that we endow with the order topology $\tau$. Then $(G,\tau)$ is a topological group.
\end{proposition}

\begin{proof}
Consider an elementary open subset $V$ of $G$ of the form 
\[ (\bigcap_{i \in [[1,u]]} ]a_i,+\infty[) \cap (\bigcap_{j \in [[1,v]]} ]-\infty,b_j)) \]
containing $ef \in G$, where $e,f \in G$, and without loss of generality, $u,v \in \{0,1\}$, because $G$ is totally ordered.
For all $i \in [[1,u]]$, we have $a_i < ef$, and thus, $e^{-1}a_i < f$. 
\begin{itemize}
\item If $]e^{-1}a_i,f[ \neq \emptyset$, let $\alpha^2_i \in ]e^{-1}a_i,f[$ and $\alpha^1_i = a_i (\alpha^2_i)^{-1}$.
\item If $]e^{-1}a_i,f[ = \emptyset$, let $\alpha^2_i = e^{-1}a_i$ and $\alpha^1_i = a_i f^{-1}$.
\end{itemize}
Similarly, for all $j \in [[1,u]]$, $ef < b_j$, and thus, $e^{-1}b_j > f$. 
\begin{itemize}
\item If $]f,e^{-1}b_j[ \neq \emptyset$,  let $\beta^2_j \in ]f,e^{-1}b_j[$ and $\beta^1_j = b_j (\beta^2_j)^{-1}$.
\item If $]f,e^{-1}b_j[ = \emptyset$, let $\beta^2_j = e^{-1}b_j$ and $\beta^1_j = b_j f^{-1}$.
\end{itemize}
Then if
\[ W_1 := (\bigcap_{i \in [[1,u]]} ]\alpha^1_i,+\infty[) \cap (\bigcap_{j \in [[1,v]]} ]-\infty,\beta^1_j)) \]
and
\[ W_2 := (\bigcap_{i \in [[1,u]]} ]\alpha^2_i,+\infty[) \cap (\bigcap_{j \in [[1,v]]} ]-\infty,\beta^2_j)) \]
we clearly have that
\[ W_1 \cdot W_2 \subseteq V, \]
and that $W_1$ and $W_2$ are open neighborhoods of $e$ and $f$ respectively. \\
Moreover, $V^{-1}$ is clearly an open subset for the order topology containing $(ef)^{-1}$. \\
Therefore, $(G,\tau)$ is a topological group.
\end{proof}

\begin{proposition}
Let $(R,\leq)$ be a totally ordered division ring that we endow with the order topology $\tau$. Then $(R,\tau)$ is a topological ring.
\end{proposition}

\begin{proof}
$(R,+,0,\tau)$ is a topological group according to the previous result. Let $R'$ be the set of strictly positive elements of $R$. Then $(R',\cdot, 1,\leq_{|R'})$ is a totally ordered group. According to the previous result, this is also a topological group for the order topology associated with the restricted order $\leq_{|R'}$. Since $R'$ is an interval in $R$, this topology of $R'$ is the subspace topology inherited from the order topology of $R$. Thus multiplication on $R$ is continuous at each point of $R' \times R'$ and the inverse function is continuous at each point of $R'$. Clearly, these two maps can be easily shown to be continuous on strictly negative numbers as well. The remaining problem is to show that multiplication is continuous at each $(x,0)$ and $(0,y)$. \\
Let us show that multiplication is continuous at $(a,0)$ when $a > 0$. \\
Let $W$ be a neighborhood of $0$ in $R$. \\
Let $\epsilon > 0$ be small enough so that $]-\epsilon,\epsilon[ \subseteq W$. \\
Let $\eta > 0$ be small enough so that $a-\eta > 0$, and let $V_1 = ]a-\eta,a+\eta[$ and $V_2 = ]-(a+\eta)^{-1}\epsilon,(a+\eta)^{-1}\epsilon[$ \\
Let $x \in V_1$ and $y \in V_2$. Let us prove that $xy \in ]-\epsilon,\epsilon[$. First, note that $x > a-\eta > 0$. Now,
\begin{itemize}
\item If $y > 0$, then $xy > 0 > -\epsilon$ and $xy < (a+\eta)(a+\eta)^{-1}\epsilon = \epsilon$.
\item If $y \leq 0$, then $-(a+\eta)^{-1}\epsilon < y \leq 0$, so that $0 \leq -y < (a+\eta)^{-1}\epsilon$. Then, we have $0 \leq -xy < x (a+\eta)^{-1}\epsilon < (a+\eta) (a+\eta)^{-1}\epsilon = \epsilon$. Therefore $-\epsilon < xy \leq 0 < \epsilon$.
\end{itemize}
Let us now show that multiplication is continuous at $(a,0)$ when $a < 0$. \\
Let $\epsilon > 0$ be small enough so that $]-\epsilon,\epsilon[ \subseteq W$. \\
Let $\eta > 0$ be small enough so that $a+\eta < 0$, and let $V_1 = ]a-\eta,a+\eta[$ and $V_2 = ](a-\eta)^{-1}\epsilon,-(a-\eta)^{-1}\epsilon[$ \\
Let $x \in V_1$ and $y \in V_2$. Let us prove that $xy \in ]-\epsilon,\epsilon[$. First, note that $x < a+\eta < 0$. Now,
\begin{itemize}
\item If $y > 0$, then $xy < 0 < \epsilon$. Moreover, $a - \eta < x < 0$, which implies that  $0 < -x < \eta - a$, so $0 < -xy < (\eta - a)(\eta-a)^{-1}\epsilon < \epsilon$. Hence $-\epsilon < xy < 0 < \epsilon$.
\item If $y \leq 0$, then $(a-\eta)^{-1}\epsilon < y \leq 0$, so $0 \leq -y < -(a-\eta)^{-1}\epsilon$, and since $0 < -x <\eta-a$, we have $0 \leq xy < (\eta-a) (\eta - a)^{-1}\epsilon = \epsilon$. Hence $-\epsilon < 0 \leq xy \leq \epsilon$.
\end{itemize}
Similarly, multiplication is continuous at $(0,b)$ with $b \in R^*$. \\
It remains to be proven that multiplication is continuous at $(0,0)$. \\
Let $W$ be a neighborhood of $0$ in $R$. \\
Let $\epsilon > 0$ be small enough so that $]-\epsilon,\epsilon[ \subseteq W$. \\
Let $V_1 = ]-1,1[$ and $V_2 = ]-\epsilon, \epsilon[$. Let $x \in V_1$ and $y \in V_2$. \\
Thus, it is clear that $xy \in ]-\epsilon,\epsilon[ \subseteq W$. \\
In conclusion, multiplication is continuous and $(R,\tau)$ is a topological ring.
\end{proof}

\subsection{Hausdorff property for normed groups or rings}

As stated, the following proposition applies to groups. However, it can also be applied to rings. Indeed, notice that we only use the group structure, so the proposition transfers directly to the realm of rings.

\begin{proposition}
Let $(G_2,\times,1,\leq)$ be a totally ordered group endowed with Frink's interval topology. \\
Moreover, let $(G_1,\times,1)$ be a group and $f : (G_1,\times,1) \to (G_2,\times,1,\leq)$ be a norm. \\
We endow $G_1$ with the right topology induced by $f$. \\
Then $G_1$ is Hausdorff.
\end{proposition}

\begin{proof}
Since $G_2$ is totally ordered, it is a topological group in Frink's interval topology by proposition \ref{proposition-totally-ordered-group-topological-group}. Therefore, for any open neighborhood $V$ of $1$ in $G_2$, there exists an open neighborhood $W$ of $1$ in $G_2$ such that $W \cdot W \subseteq V$. \\
Now, let $a \neq b$ in $G_1$. Since $f$ is a norm, we have $f(ab^{-1}) \neq 1$. Since Frink's interval topology is locally-convex and $T_1$, there exists a convex open neighborhood $V$ of $1$ in $G_2$ such that $f(ab^{-1}) \notin V$. Let $W$ be an open neighborhood of $1$ in $G_2$ such that $W \cdot W \subseteq V$. Then
\[ B_W(a) \cap B_W(b) = \{ x \in G_1 : f(xa^{-1}),f(xb^{-1}) \in W \} \subseteq \{ x \in G_1 : f(ab^{-1}) \in V\} = \emptyset \]
since $1 \leq f(ab^{-1}) \leq f(ax^{-1}) \cdot f(xb^{-1}) =  f(xa^{-1}) \cdot f(xb^{-1}) \in W \cdot W \subseteq V$ for any $x \in B_W(a) \cap B_W(b)$, and $V$ is convex. \\
Therefore, $G_1$ is Hausdorff.
\end{proof}

\section{Invertibility in nonassociative rings via oriented geometric series}

The following definition is important in this section.
 
\begin{definition}
Let $(R,+,0,\times,1,\leq)$ be a partially ordered unitary nonassociative ring such that $1 \geq 0$. Let $x > 0$. \\ 
If $\exists y > 0 : xy \geq 1$, we say that $x$ admits a \textbf{right-sup-almost-inverse} $y$. \\ 
If $\exists y > 0 : yx \geq 1$, we say that $x$ admits a \textbf{left-sup-almost-inverse} $y$.
\end{definition}

\subsection{Invertibility in partially ordered nonassociative rings}

The importance of the following theorem resides in the fact that there exist nonobvious examples of rings satisfying its conditions, as stated in example \ref{example-right-invertible} (by nonobvious, we mean different from $\mathbb{R}$ and $\mathbb{Z}$). 

\begin{theorem}
\label{theorem-right-invertible}
Let $(R,+,0,\times,1,\leq)$ be a partially ordered nonassociative ring with unit. \\
For all $x \in R$, define by induction $x^{\to 0} = 1$ and $x^{\to n+1} = xx^{\to n}$. \\
Suppose that :
\begin{enumerate}
\item $1 \geq 0$.
\item $R$ is monotone $\sigma$-complete.
\item Any $x \in ]0,1]$ has a right-sup-almost-inverse.
\item $\forall x \in [0,1[ : \inf\limits_{n \in \mathbb{N}} \{x^{\to n}\} = 0$.
\end{enumerate}
Then we have: 
\begin{itemize}
\item $\forall x \in [0,1[ : \sup\limits_{n \in \mathbb{N}} \{\sum\limits_{k=0}\limits^n x^{\to k}\}$ exists, which we denote by $\sum\limits_{n \in \mathbb{N}} x^{\to n}$ since it is an actual limit (see lemma \ref{lemma-increasing-sequence-converges-to-sup-interval-topology}).
\item $\forall x \in [0,1[ : x \sum\limits_{n \in \mathbb{N}} x^{\to n} = \sum\limits_{n \in \mathbb{N}} x^{\to n+1} =: \sum\limits_{n \in \mathbb{N}^*} x^{\to n}$,
\item $\forall x \in ]0,1] : x \text{ admits a right inverse } x_R^{-1} = \sum\limits_{n \in \mathbb{N}} (1-x)^{\to n} = 1 + a \geq 1$ where $a := \sum\limits_{n \in \mathbb{N}^*} (1-x)^{\to n} \geq 0$.
\end{itemize}
\end{theorem}

\begin{proof}
\begin{itemize}
\item Let $x \in [0,1[$. Let $c > 0$ such that $(1-x)c \geq 1$. Let, for all $n \in \mathbb{N}$, $s_n := \sum\limits_{k=0}\limits^n x^{\to k}$. Let us show by induction that $\forall n \in \mathbb{N} : s_n \leq c$. We have $s_0 = 1 \leq (1-x)c \leq c$. Suppose that we have $s_n \leq c$ for a certain $n \in \mathbb{N}$. Then $s_{n+1} = xs_n + 1 \leq xc + 1 \leq c$.  Hence by monotone $\sigma$-completeness $\sum\limits_{n \in \mathbb{N}} x^{\to n} := \sup\limits_{n \in \mathbb{N}} \{s_n\}$ exists.
\item Let $x \in [0,1[$. Let, for all $n \in \mathbb{N}$, $s_n := \sum\limits_{k=0}\limits^n x^{\to k}$. We have $\forall n \in \mathbb{N} : s_n \leq \sum\limits_{n \in \mathbb{N}} x^{\to n}$. Therefore, $\forall n \in \mathbb{N} : \sum\limits_{k = 0}\limits^{n} x^{\to k+1} = xs_n \leq x \sum\limits_{n \in \mathbb{N}} x^{\to n}$, and so $\sum\limits_{n \in \mathbb{N}} x^{\to n+1} \leq x \sum\limits_{n \in \mathbb{N}} x^{\to n}$ by the definition of the supremum. Conversely, we need to show that $x \sum\limits_{n \in \mathbb{N}} x^{\to n} \leq \sum\limits_{n \in \mathbb{N}} x^{\to n+1} = (\sum\limits_{n \in \mathbb{N}} x^{\to n}) - 1$ (notice that for all bounded above countable subset $A \subseteq R$ and element $a \in R$ we have $\sup(A+a)=\sup(A)+a$). This is equivalent to $(1-x) (\sum\limits_{n \in \mathbb{N}} x^{\to n}) \geq 1$. But we have $\forall m \in \mathbb{N}^* : (1-x) (\sum\limits_{n \in \mathbb{N}} x^{\to n}) \geq 1 - x^m$. Passing to the supremum and using that $\inf\limits_{m \in \mathbb{N}} \{x^{\to m}\} = 0$, we have the desired inequality $(1-x) (\sum\limits_{n \in \mathbb{N}} x^{\to n}) \geq 1$.
\item We have $1 = (1 - (1-x)) \left[ \sum\limits_{n \in \mathbb{N}} (1-x)^{\to n} \right] = xx_R^{-1}$ where $x_R^{-1} = \sum\limits_{n \in \mathbb{N}} (1-x)^{\to n} = 1 + a \geq 1$ where $a := \sum\limits_{n \in \mathbb{N}^*} (1-x)^{\to n} \geq 0$.
\end{itemize}
\end{proof}

\begin{corollary}
\label{corollary-invertible}
Let $(R,+,0,\times,1,\leq)$ be a partially ordered unitary nonassociative ring. \\
For all $x \in R$, define by induction $x^{\to 0} = 1$ and $x^{\to n+1} = xx^{\to n}$. \\
For all $x \in R$, define by induction $x^{\leftarrow 0} = 1$ and $x^{\leftarrow n+1} = x^{\leftarrow n}x$. \\
Suppose that :
\begin{enumerate}
\item $1 \geq 0$,
\item $R$ is monotone $\sigma$-complete,
\item Any $x \in ]0,1]$ has a right-sup-almost-inverse,
\item Any $x \in ]0,1]$ has a left-sup-almost-inverse,
\item $\forall x \in [0,1[ : \inf\limits_{n \in \mathbb{N}} \{x^{\to n}\} = 0$,
\item $\forall x \in [0,1[ : \inf\limits_{n \in \mathbb{N}} \{x^{\leftarrow n}\} = 0$,
\item $\forall x \in ]0,1[ :$ ($x$ is left and right invertible) $\Rightarrow (x_L^{-1} = x_R^{-1})$ (where $x_L^{-1}$ is a left-inverse of $x$, and $x_R^{-1}$ a right inverse),
\end{enumerate}
then for all $x \in ]0,1]$, $x$ is invertible with inverse $\sum\limits_{n \in \mathbb{N}} (1-x)^{\leftarrow n}$.
\end{corollary}

\begin{proof}
In an analogous way to the proof of theorem \ref{theorem-right-invertible}, we have $1 = \left[ \sum\limits_{n \in \mathbb{N}} (1-x)^{\leftarrow n} \right] (1-(1-x)) = x_L^{-1}x$, where $x_L^{-1} = \sum\limits_{n \in \mathbb{N}} (1-x)^{\leftarrow n} \geq 1$. This means that $x$ has also a left inverse $x_L^{-1} = x_R^{-1}$. Therefore $x$ is invertible with unique inverse $\sum\limits_{n \in \mathbb{N}} (1-x)^{\to n} = \sum\limits_{n \in \mathbb{N}} (1-x)^{\leftarrow n}$.
\end{proof}

\begin{remark}
If : 
\begin{enumerate}
\item $1 \geq 0$,
\item $R$ is power-associative,
\item $R$ is monotone $\sigma$-complete,
\item Any $x \in ]0,1]$ has a right-sup-almost-inverse or a left-sup-almost-inverse,
\item $\forall x \in [0,1[ : \inf\limits_{n \in \mathbb{N}} \{x^n\} = 0$,
\end{enumerate}
then for all $x \in ]0,1]$, $x$ is invertible with inverse $\sum\limits_{n \in \mathbb{N}} (1-x)^n$.
\end{remark}

\begin{example}
\label{example-right-invertible}
For instance, theorem \ref{theorem-right-invertible} holds in the rings $\mathbb{R}$ and $\mathbb{Z}$ with their usual total orderings, but also for $\mathbb{R}[X]$ and $\mathbb{Z}[X]$ with the lexicographic partial ordering $P = \sum\limits_{k=0}^d c_kX^k \geq 0$ iff $P=0$ or $c_d > 0$.
\end{example}

\begin{remark}
\textbf{Optimality of the result}: the fact that in $\mathbb{Z}$, 1 is the only invertible element in the interval $]0,+\infty$[, shows that theorem \ref{theorem-right-invertible} is not very far from being optimal : an interval of the form $]0,a]$ with $a \geq 2$ cannot entirely contain invertible elements in a partially ordered ring in general. An interesting question would be to prove that an interval of the form $]0,a]$ with $a > 1$ cannot entirely contain invertible elements in a partially ordered ring in general, because that would be optimal. 
\end{remark}

\begin{remark}
\textbf{Necessity of the hypotheses}: the following examples show the necessity of the hypotheses of theorem \ref{theorem-right-invertible}, or more explicitly, that there exist rings that fail to satisfy one of its conditions and for which the conclusion of the theorem is also false. 
\begin{itemize}
\item Consider $\mathbb{Q}^2$ with the lexicographic order $(a,b) \geq (0,0)$ iff $a>0$ or ($a=0$ and $b\geq0$). This ring is not monotone $\sigma$-complete (also $\inf\limits_{n \in \mathbb{N}}\{x^n\}$ does not exist for $x \in [(0,0),(1,1)[)$, has a sup-almost-inverse for each $x \in ](0,0),(1,1)]$, but does not have inverses for all $x \in ](0,0),(1,1)]$.
\item Obviously, a ring not admitting right-sup-almost-inverses for elements $x \in ]0,1]$ cannot admit right inverses. Consider for instance $\mathbb{R}[X]$ with the antilexicographic partial ordering $P = \sum\limits_{k=v}^d c_kX^k \geq 0$ iff $P=0$ or $c_v > 0$. This ring satisfies all the other conditions.
\item Consider $\mathbb{R}^2$ with the componentwise sum and product, and the partial ordering $(a,b) \geq (0,0)$ iff ($a \geq 0$ and $b \geq 0$). This ring satisfies all the conditions except $\forall x \in [(0,0),(1,1)[ : \inf\limits_{n \in \mathbb{N}} \{x^n\} = 0$.
\end{itemize}
\end{remark}

\subsection{Invertibility in seminormed Hausdorff Cauchy-complete weak-quasi-topological nonassociative rings}

\begin{theorem}
\label{theorem-f(1-x)<1}
Let $(R_1,+,0,\times,1)$ be a unitary nonassociative ring and $(R_2,+,0,\times,1,\leq)$ be a partially ordered unitary nonassociative ring that we endow with the interval topology. \\
We denote the elements of $R_1$ by Latin letters and the elements of $R_2$ by Greek letters. \\
For all $x \in R_1$, define by induction: $x^{\to 0} = 1$ and $\forall n \in \mathbb{N} : x^{\to n+1} = xx^{\to n}$. \\
For all $\xi \in R_2$, define by induction: $\xi^{\to 0} = 1$ and $\forall n \in \mathbb{N} : \xi^{\to n+1} = \xi \xi^{\to n}$. \\
Suppose that :
\begin{enumerate}
\item $1 \geq 0$,
\item $R_2$ is monotone $\sigma$-complete,
\item Any $\xi \in ]0,1]$ has a right-sup-almost-inverse, 
\item $\forall \xi \in [0,1[ : \inf\limits_{n \in \mathbb{N}} \{\xi^{\to n}\} = 0$.
\end{enumerate}
Let $f : (R_1,+,0,\times,1) \to (R_2,+,0,\times,1,\leq)$ be a seminorm. \\
Suppose that $R_1$ is Hausdorff sequentially Cauchy-complete for the topology induced by the seminorm $f$ (we endow $R_2$ with Frink's interval topology). \\
Then for all $x \in R_1$ such that $f(1-x) < 1$, $x$ admits a right-inverse in $R_1$.
\end{theorem}

\begin{proof}
By assumption, we have $0 \leq f(1-x) < 1$. Therefore, by theorem \ref{theorem-right-invertible}, $\sum\limits_{n \in \mathbb{N}} f(1-x)^{\to n}$ exists. Let $v_n = \sum\limits_{k = 0}^n f(1-x)^{\to k}$ for all $n \in \mathbb{N}$. Since the sequence of remainders of the series $\sum\limits_{k = 0}^n f(1-x)^{\to k}$ is non-negative and decreasing to $0$, it is easily seen that $(v_n)_{n \in \mathbb{N}}$ is a Cauchy sequence in $R_2$ for the interval topology. \\
Let $u_n = \sum\limits_{k = 0}^n (1-x)^{\to k}$ for all $n \in \mathbb{N}$. By the triangle inequality and submultiplicativity, $(u_n)$ is a Cauchy sequence in $R_1$ for the topology induced by the seminorm $f$. Since $R_1$ is Hausdorff sequentially Cauchy-complete for this topology, $(u_n)_{n \in \mathbb{N}}$ converges to a unique limit, which we denote by $\sum\limits_{n \in \mathbb{N}} (1-x)^{\to n}$. \\
We would like to show that $x$ admits a right-inverse $x_R^{-1} := \sum\limits_{n \in \mathbb{N}} (1-x)^{\to n}$. This follows from the following computation:
\begin{align*}
x x_R^{-1} &= (x-1) x_R^{-1} + x_R^{-1} \\
&= (x-1) \lim_n u_n + x_R^{-1} \\
&= \lim_n (x-1)u_n + x_R^{-1} \\
&= \lim_n (1 - u_{n+1}) + x_R^{-1} \\
&= 1 - x_R^{-1} + x_R^{-1} \\
&= 1
\end{align*}
where we have used the fact that the limit of $(u_n)_{n \in \mathbb{N}}$ is unique and proposition \ref{proposition-phia-continuous}.
\end{proof}

\begin{corollary}
\label{corollary-f(1-x)<1}
Let $(R_1,+,0,\times,1)$ be a unitary nonassociative ring and $(R_2,+,0,\times,1,\leq)$ be a partially ordered unitary nonassociative ring that we endow with the interval topology. \\
We denote the elements of $R_1$ by Latin letters and the elements of $R_2$ by Greek letters. \\
For all $x \in R_1$, define by induction: $x^{\to 0} = 1$ and $\forall n \in \mathbb{N} : x^{\to n+1} = xx^{\to n}$. \\
For all $x \in R_1$, define by induction $x^{\leftarrow 0} = 1$ and $x^{\leftarrow n+1} = x^{\leftarrow n}x$. \\
For all $\xi \in R_2$, define by induction: $\xi^{\to 0} = 1$ and $\forall n \in \mathbb{N} : \xi^{\to n+1} = \xi \xi^{\to n}$. \\
For all $\xi \in R_2$, define by induction: $\xi^{\leftarrow 0} = 1$ and $\forall n \in \mathbb{N} : \xi^{\leftarrow n+1} = \xi \xi^{\leftarrow n}$. \\
Suppose that :
\begin{enumerate}
\item $1 \geq 0$ in $R_2$,
\item $R_2$ is monotone $\sigma$-complete,
\item Any $\xi \in ]0,1]$ has a right-sup-almost-inverse, 
\item Any $\xi \in ]0,1]$ has a left-sup-almost-inverse,
\item $\forall \xi \in [0,1[ : \inf\limits_{n \in \mathbb{N}} \{\xi^{\to n}\} = 0$,
\item $\forall \xi \in [0,1[ : \inf\limits_{n \in \mathbb{N}} \{\xi^{\leftarrow n}\} = 0$,
\end{enumerate}
Let $f : (R_1,+,0,\times,1) \to (R_2,+,0,\times,1,\leq)$ be a seminorm. \\
Moreover, suppose that:
\begin{enumerate}
\item $R_1$ is Hausdorff sequentially Cauchy-complete for the topology induced by the seminorm $f$ (we endow $R_2$ with Frink's interval topology),
\item $\forall x \in R_1$ : ($x$ is left and right invertible) $\Rightarrow (x_L^{-1} = x_R^{-1})$ (where $x_L^{-1}$ is a left-inverse of $x$, and where $x_R^{-1}$ is a right inverse),
\end{enumerate}
Then for all $x \in R_1$ such that $f(1-x) < 1$, $x$ is invertible in $R_1$.
\end{corollary}

\begin{proof}
Analogously to the proof of theorem \ref{theorem-f(1-x)<1}, we have $1 = \left[ \sum\limits_{n \in \mathbb{N}} (1-x)^{\leftarrow n} \right] (1-(1-x)) = x_L^{-1}x$, where $x_L^{-1} = \sum\limits_{n \in \mathbb{N}} (1-x)^{\leftarrow n}$. This means that $x$ also has a left inverse $x_L^{-1} = x_R^{-1}$. Therefore $x$ is invertible with unique inverse $\sum\limits_{n \in \mathbb{N}} (1-x)^{\to n} = \sum\limits_{n \in \mathbb{N}} (1-x)^{\leftarrow n}$.
\end{proof}

\begin{remark}
\label{remark-theorem-2}
If : 
\begin{enumerate}
\item $1 \geq 0$ in $R_2$,
\item $R_1$ and $R_2$ are power-associative,
\item $\forall \xi \in [0,1[ : \inf\limits_{n \in \mathbb{N}} \{\xi^n\} = 0$,
\item Any $\xi \in ]0,1]$ has a right-sup-almost-inverse or a left-sup-almost-inverse,
\item $R_1$ is Hausdorff sequentially Cauchy-complete for the topology induced by the seminorm $f$ (we endow $R_2$ with Frink's interval topology),
\end{enumerate}
then for all $x \in ]0,1]$, $x$ is invertible with inverse $\sum\limits_{n \in \mathbb{N}} (1-x)^n$.
\end{remark}

\subsection{Two additional independent lemmas}

\begin{lemma}
Let $(R,+,0,\times,1,\leq)$ be an Archimedean partially ordered unitary nonassociative ring of characteristic $0$ such that $1 > 0$. Let $x > 0$ such that
\[ \forall y \in \mathbb{N}x : y \text{ is comparable to } 1. \]
Then $x$ has a left-and-right-sup-almost-inverse $y$.
\end{lemma}

\begin{proof}
Indeed, for any such $x$, there exists by the Archimedean property an $n \in \mathbb{N}$ such that $nx \geq 1$, and we can take $y = \underbrace{n.1}_{> 0} \in R$.
\end{proof}

\begin{lemma}
Let $(R,+,0,\times,1,\leq)$ be a monotone $\sigma$-complete partially ordered associative ring. \\
Then for any invertible element $x \in R$ such that $x \in ]0,1[$ and $1-x$ is not a zero divisor, we have $\inf\limits_{n \in \mathbb{N}} \{x^n\} = 0$.
\end{lemma}

\begin{proof}
Indeed, let $x \in R$ be an invertible element such that $x \in ]0,1[$ and $1-x$ is not a zero divisor. We have $\forall m \in \mathbb{N} :  \inf_{n \in \mathbb{N}} \{x^n\} \leq x^{m+1}$, and thus, since $x$ is invertible in $R$ and $x^{-1} > 0$, $\forall m \in \mathbb{N} : x^{-1} \inf_{n \in \mathbb{N}} \{x^n\} \leq x^{m}$, which implies that $x^{-1} \inf_{n \in \mathbb{N}} \{x^n\} \leq \inf_{n \in \mathbb{N}} \{x^n\}$, and so $\inf_{n \in \mathbb{N}} \{x^n\} \leq x \inf_{n \in \mathbb{N}} \{x^n\}$. Let $a := \inf\limits_{n \in \mathbb{N}} \{x^n\}$. We have $a \leq xa \leq a$. Therefore $xa = a$, which is equivalent to $(x-1)a=0$ and so $a = 0$ since $1-x$ is not a zero divisor.
\end{proof}

\section*{Conflict of interest}
On behalf of all authors, the corresponding author states that there is no conflict of interest.  

\section*{Data availability statement}
No data are associated with this article. \\ \\

\nocite{*}
\bibliographystyle{plain}
\bibliography{references}

\begin{thebibliography}{10}

\bibitem{Amadio-Curien}
R.~M. Amadio and P.-L. Curien.
\newblock {\em Domains and lambda-calculi}.
\newblock Number~46. Cambridge University Press, 1998.

\bibitem{Arnautov-Glavatsky-Mikhalev}
V.~I. Arnautov, S.~T. Glavatsky, and A.~V. Mikhalev.
\newblock {\em Introduction to the Theory of Topological Rings and Modules}.
\newblock Books in Soils, Plants, and the Environment. Taylor \& Francis, 1996.

\bibitem{Birkhoff}
G.~Birkhoff.
\newblock {\em Lattice theory}, volume~25.
\newblock American Mathematical Soc., 1940.

\bibitem{Branga-Olaru}
A.~N. Branga and I.~M. Olaru.
\newblock Cone metric spaces over topological modules and fixed point theorems
  for lipschitz mappings.
\newblock {\em Mathematics}, 8(5):724, 2020.

\bibitem{Chapman-Fontana-Geroldinger-Olberding}
S.~T. Chapman, M.~Fontana, A.~Geroldinger, and B.~Olberding.
\newblock Multiplicative ideal theory and factorization theory.
\newblock {\em Proceedings in Mathematics and Statistics}, 170, 2016.

\bibitem{Clark-Diepeveen}
P.~L. Clark and N.~J. Diepeveen.
\newblock Absolute convergence in ordered fields.
\newblock {\em The American Mathematical Monthly}, 121(10):909--916, 2014.

\bibitem{Deveau-Teismann}
M.~Deveau and H.~Teismann.
\newblock 72+42: Characterizations of the completeness and archimedean
  properties of ordered fields.
\newblock {\em Real Anal. Exchange}, 39(2):261--304, 2013/2014.

\bibitem{Drazin1958}
M.~P. Drazin.
\newblock Pseudo-inverses in associative rings and semigroups.
\newblock {\em The American mathematical monthly}, 65(7):506--514, 1958.

\bibitem{Drazin2012}
M.~P. Drazin.
\newblock A class of outer generalized inverses.
\newblock {\em Linear algebra and its applications}, 436(7):1909--1923, 2012.

\bibitem{Frink1942}
O.~Frink.
\newblock Topology in lattices.
\newblock {\em Transactions of the American Mathematical Society}, 51:569--582,
  1942.

\bibitem{Frink1954}
O.~Frink.
\newblock Ideals in partially ordered sets.
\newblock {\em The American Mathematical Monthly}, 61(4):223--234, 1954.

\bibitem{Garcia-Pacheco}
F.~J. Garc{\'\i}a-Pacheco.
\newblock Regularity in topological modules.
\newblock {\em Mathematics}, 8(9):1580, 2020.

\bibitem{Garcia-Pacheco-Miralles-Murillo-Arcila}
F.~J. Garc{\'\i}a-Pacheco, A.~Miralles, and M.~Murillo-Arcila.
\newblock Invertibles in topological rings: a new approach.
\newblock {\em Revista de la Real Academia de Ciencias Exactas, F{\'\i}sicas y
  Naturales. Serie A. Matem{\'a}ticas}, 116(1):38, 2022.

\bibitem{Garcia-Pacheco-Moreno-Frias-Murillo-Arcila}
F.~J. Garc{\'\i}a-Pacheco, M.~A. Moreno-Fr{\'\i}as, and M.~Murillo-Arcila.
\newblock Topological ordered rings and measures.
\newblock {\em Results in Mathematics}, 78(6):212, 2023.

\bibitem{Garcia-Pacheco-Martinez}
F.~J. Garc{\'\i}a-Pacheco and S.~S{\'a}ez-Mart{\'\i}nez.
\newblock Normalizing rings.
\newblock {\em Banach Journal of Mathematical Analysis}, 14(3):1143--1176,
  2020.

\bibitem{Gierz-Hofmann-Keimel-Lawson-Mislove-Scott}
G.~Gierz, K.~H. Hofmann, K.~Keimel, J.~D. Lawson, M.~Mislove, and D.~S. Scott.
\newblock {\em A compendium of continuous lattices}.
\newblock Springer Science \& Business Media, 2012.

\bibitem{Glass}
A.~M.~W. Glass.
\newblock {\em Partially ordered groups}, volume~7.
\newblock World Scientific, 1999.

\bibitem{Hall}
J.~F. Hall.
\newblock Completeness of ordered fields.
\newblock {\em arXiv preprint arXiv:1101.5652}, 2011.

\bibitem{Handelman}
D.~Handelman.
\newblock Rings with involution as partially ordered abelian groups.
\newblock {\em The Rocky Mountain Journal of Mathematics}, 11(3):337--381,
  1981.

\bibitem{Henriksen}
M.~Henriksen.
\newblock A survey of f-rings and some of their generalizations.
\newblock In {\em Ordered Algebraic Structures: Proceedings of the Cura\c{c}ao
  Conference. Sponsored by the Caribbean Mathematics Foundation, June 26--30,
  1995.}, pages 1--26. Springer Netherlands, 1997.

\bibitem{Kantrowitz-Neumann2015}
R.~Kantrowitz and M.~Neumann.
\newblock Another face of the archimedean property.
\newblock {\em The College Mathematics Journal}, 46(2):139--141, 2015.

\bibitem{Kantrowitz-Neumann2017}
R.~Kantrowitz and M.~Neumann.
\newblock Normed algebras and the geometric series test.
\newblock {\em Surveys in Mathematics and its Applications}, 12:203--217, 2017.

\bibitem{Kaplansky}
I.~Kaplansky.
\newblock Topological rings.
\newblock {\em American Journal of Mathematics}, 69(1):153--183, 1947.

\bibitem{Kopperman-Pajoohesh}
R.~Kopperman and H.~Pajoohesh.
\newblock Representing topologies using partially ordered semigroups.
\newblock {\em Topology and its applications}, 249:135--149, 2018.

\bibitem{Lam}
T.-Y. Lam.
\newblock {\em A first course in noncommutative rings}, volume 131.
\newblock Springer, 1991.

\bibitem{Mathews-Anderson}
J.~C. Mathews and R.~F. Anderson.
\newblock A comparison of two modes of order convergence.
\newblock {\em Proceedings of the American Mathematical Society},
  18(1):100--104, 1967.

\bibitem{Matsumura}
H.~Matsumura.
\newblock {\em Commutative ring theory}.
\newblock Number~8. Cambridge university press, 1989.

\bibitem{Mcshane}
E.~J. McShane.
\newblock {\em Order-preserving maps and integration processes}.
\newblock Number~31. Princeton University Press, 1954.

\bibitem{Meyer-Nieberg}
P.~Meyer-Nieberg.
\newblock {\em Banach lattices}.
\newblock Springer Science \& Business Media, 2012.

\bibitem{Olivos-Soto-Mansilla}
E.~Olivos, H.~Soto, and A.~Mansilla.
\newblock Metrizability of totally ordered groups of infinite rank and their
  completions.
\newblock {\em Bull. Belg. Math. Soc. - Simon Stevin}, 14(5):969--977, 2007.

\bibitem{Ore}
O.~Ore.
\newblock Linear equations in non-commutative fields.
\newblock {\em Annals of Mathematics}, 32(3):463--477, 1931.

\bibitem{Pajoohesh}
H.~Pajoohesh.
\newblock Representations of bornologies.
\newblock {\em Applied General Topology}, 23(1):17--30, 2022.

\bibitem{Propp}
J.~Propp.
\newblock Real analysis in reverse.
\newblock {\em The American Mathematical Monthly}, 120(5):392--408, 2013.

\bibitem{Rao}
K.P.S.~Bhaskara Rao.
\newblock {\em Theory of generalized inverses over commutative rings},
  volume~17.
\newblock CRC Press, 2002.

\bibitem{Rennie}
B.~C. Rennie.
\newblock Lattices.
\newblock {\em Proceedings of the London Mathematical Society}, 2(1):386--400,
  1950.

\bibitem{Schaefer}
H.~H. Schaefer.
\newblock {\em Banach Lattices and Positive Operators}.
\newblock Springer Berlin, Heidelberg, 1974.

\bibitem{Schafer}
R.~D. Schafer.
\newblock {\em An introduction to nonassociative algebras}.
\newblock Courier Dover Publications, 2017.

\bibitem{Steinberg}
S.~A. Steinberg.
\newblock {\em Lattice-ordered rings and modules}.
\newblock Springer, 2010.

\bibitem{Sun-Li-Fan}
T.~Sun, Q.~Li, and N.~Fan.
\newblock {MN}-convergence and lim-inf{$_M$}-convergence in partially ordered
  sets.
\newblock {\em Open Mathematics}, 16(1):1077--1090, 2018.

\bibitem{Teismann}
H.~Teismann.
\newblock Toward a more complete list of completeness axioms.
\newblock {\em The American Mathematical Monthly}, 120(2):99--114, 2013.

\bibitem{Ward}
A.~J. Ward.
\newblock On relations between certain intrinsic topologies in partially
  ordered sets.
\newblock In {\em Mathematical Proceedings of the Cambridge Philosophical
  Society}, volume~51, pages 254--261. Cambridge University Press, 1955.

\bibitem{Warner}
S.~Warner.
\newblock {\em Topological rings}.
\newblock Elsevier, 1993.

\bibitem{Wolk}
E.~S. Wolk.
\newblock Order-compatible topologies on a partially ordered set.
\newblock {\em Proceedings of the American Mathematical Society},
  9(4):524--529, 1958.

\end{thebibliography}

\Addresses

\end{document}